\documentclass[reqno,10pt,a4paper,oneside]{amsart}

\usepackage{uniform-kan-prelude}
\usepackage{formatting-christian}

\usepackage{tikz-cd}

\title{Free Monad Sequences and Extension Operations}
\date{
\\
2016-09-18: initial version
\\
2024-09-09: added cube diagram in Section 3.4
\\
2025-04-10: added version history, arXiving}

\begin{document}

\begin{abstract}
In the first part of this article, we give an analysis of the free monad sequence in non-cocomplete categories, with the needed colimits explicitly parametrized.
This enables us to state a more finely grained functoriality principle for free monad and monoid sequences.

In the second part, we deal with the problem of functorially extending via pullback squares a category of maps along the category of coalgebras of an algebraic weak factorization system.
This generalizes the classical problem of extending a class of maps along the left class of a weak factorization system in the sense of pullback squares where the vertical maps are in the chosen class and the bottom map is in the left class.
Such situations arise in the context of model structures where one might wish to extend fibrations along trivial cofibrations.
We derive suitable conditions for the algebraic analogue of weak saturation of the extension problem, using the results of the first part to reduce the technical burden.
\end{abstract}

\author{Christian Sattler}

\maketitle

\tableofcontents

\newcommand{\ConfWP}[1]{\operatorname{ConfMnd}_{\text{wp}}^{{#1}}}
\newcommand{\ConfP}[1]{\operatorname{ConfMnd}_{\text{p}}^{{#1}}}
\newcommand{\ConfPM}[1]{\operatorname{ConfMon}_{\text{p}}^{{#1}}}
\newcommand{\Conf}{\operatorname{Conf}}
\newcommand{\FF}{\mathbf{FF}}

\section{Introduction}

Let $\cal{E}$ be a locally presentable category.
Let $(\cal{L}, \cal{R})$ be a weak factorization system cofibrantly generated by a set $\cal{I} \subseteq \cal{E}^\to$ of arrows, for example the (trivial cofibration, fibration) weak factorization system of a model structure.
Let further $\cal{F} \subseteq \cal{E}^\to$ be a class of arrows, for example the class of fibrations.
In certain situations, one wishes to extend maps in $\cal{F}$ along maps in $\cal{L}$, in the following sense: given a map $A \to B$ in $\cal{L}$ and a map $X \to A$ in $\cal{F}$, one wishes to find a pullback square
\begin{equation} \label{extension-square}
\begin{gathered}
\xymatrix{
  X
  \ar@{.>}[r]
  \ar[d]^{\in \cal{F}}
  \pullback{dr}
&
  Y
  \ar@{.>}[d]^{\in \cal{F}}
\\
  A
  \ar[r]
&
  B
}
\end{gathered}
\end{equation}
with $Y \to B$ again in $\cal{F}$.
We say that $A \to B$ admits \emph{extension} if this is the case for all $X \to A$ in $\cal{F}$.
Let us take for granted that maps in $\cal{I}$ admit extension.
One then needs to show that the class of maps admitting extension is weakly saturated.

If $\cal{E}$ is a presheaf category and $\cal{F}$ is a class of local fibrations with size-bounded fibers in the sense of Cisinksi~\cite{cisinski-univalence}, then $\cal{E}$ admits a classifier $U$ for maps in $\cal{F}$.
If furthermore all maps in $\cal{I}$ are objectwise decidable monomorphisms, one may phrase the extension problem~\eqref{extension-square} as a lifting problem
\[
\xymatrix{
  A
  \ar[r]
  \ar[d]
&
  U
\\
  B
  \ar@{.>}[ur]
}
\]
against $U$.
Weak saturation is then automatic.
Unfortunately, an inspection of~\cite{cisinski-univalence} reveals the restriction to objectwise decidable monomorphisms in $\cal{I}$ as essential.

More generally, one may separately verify that the class of maps admitting extension is closed under the saturation operations of coproduct, pushout, transfinite composition, and retract.
Under the assumption that $\cal{F}$ is closed under colimits in $\cal{E}^\to$ \notec{colimits in $\cat{E}^\to_{\cart}$?}, this boils down to showing that the colimits of the saturation operations are van Kampen in the sense of~\cite{heindel:van-kampen-universal}.
This is the case if $\cal{I}$ consists of monomorphisms and $\cal{E}$ is a topos, or more generally an extensive, adhesive, and exhaustive category.
Having eliminated the use of a classifier, no counterpart to objectwise decidability is needed here.

\addvspace{0.3cm}

The goal of this article is to provide a result analogous to the preceding paragraph for the algebraic setting, starting with an algebraic weak factorization system $(L, R)$ on $\cal{E}$ cofibrantly generated by a small category $u \co \cal{I} \to \cal{E}^\to$ of arrows in the sense of Garner's algebraic variant of the small object argument~\cite{garner:small-object-argument}, for example the (trivial cofibration, fibration) algebraic weak factorization system of an algebraic model structure.
Instead of a class $\cal{F}$, we consider a category $v \co \cal{F} \to \cal{E}^\to$, for example the category of fibrations $\cal{F} \defeq \liftr{\cal{I}}$ consisting of maps together with a coherent family of lifts against maps in $\cal{I}$.
Extension~\eqref{extension-square} along $A \to B$ now involves finding a lift of $Y \to B$ to $\cal{F}$, given a map $X \to A$ with a lift to $\cal{F}$.
Furthermore, we require the square, read horizontally as a morphisms of arrows, to lift to a morphism in $\cal{F}$.
Altogether, an \emph{extension operation} of $\cal{F}$ along a category of arrows $\cal{J} \to \cal{E}^\to$ should give rise to a functor $\cal{F} \times_{\cal{E}} \cal{J} \to \cal{F}^\to$.
Our aim then is to get an extension operation of $\cal{F}$ along $\coalg(L)$ from an extension operation of $\cal{F}$ along $\cal{I}$ subject to assumptions to be identified.

To see how $\coalg(L)$ arises from $\cal{I}$, we have to inspect the algebraic small object argument~\cite{garner:small-object-argument}.
One first constructs a single-step functorial factorization $(L_0, R_0)$ freely on $\cal{I}$ where $R_0$ does not yet come with a monad multiplication.
One then uses Kelly's~\cite{kelly:transfinite} free monad sequence on the pointed endofunctor $R_0$ to construct the monad part $R$ of the final algebraic weak factorization system $(L, R)$.
Finally, arrows lifting to $\coalg(L)$ arise functorially as retracts of $L$, \ie the domain part of the unit of $R$.
Letting $\eta_0$ and $\eta$ denote the respective units of $R_0$ and $R$, the most intricate part consists of going from an extension operation of $\cal{F}$ along ${\dom} \cc \eta_0$ to an extension operation along ${\dom} \cc \eta$.

Differing from the standard small object argument, the algebraic variant follows up each step of freely adding fillers by quotienting out fillers already added in previous steps.
In the explication~\cite[Section~17]{kelly:transfinite} of the free monad sequence on a pointed endofunctor, this is done using certain coequalizers.
This is also the description of~\cite[Paragraph~4.16]{garner:small-object-argument} in the context of the free monoid sequence.
Unfortunately, these coequalizers are not generally van Kampen (though this can be remedied by enlarging the shape of the colimit diagram by an initial object), and so the explication is not of direct use to us.
However, the reduction to the free monad sequence on a wellpointed endofunctor~\cite[Section~14]{kelly:transfinite} has the correct van Kampen colimits under suitable assumptions.
The extension operation of $\cal{F}$ along $\coalg(L)$ can then again be built up step-by-step.

The basic strategy outlined above can be made less tedious through an abstraction.
In going from extending along ${\dom} \cc \eta_0$ to extending along ${\dom} \cc \eta$, we do in fact not need to manually push extension through each step of the free monad sequence, we merely have to use certain functoriality properties.
An extension operation of $\cal{F}$ along ${\dom} \cc \eta_0$ gives rise to a pointed endofunctor $R_0'$ on $\cal{F}_{\cart} \times_{\cal{E}} \cal{E}^\to$ that extends $R_0$.
Here, $\cal{F}_{\cart}$ denotes the wide subcategory of $\cal{F}$ with maps restricted to pullback squares.
The forgetful functor $\cal{F}_{\cart} \times_{\cal{E}} \cal{E}^\to \to \cal{E}^\to$ will not generally preserve all connected colimits, but enough for the free monad sequence for $R_0'$ to map to the free monad sequence for $R_0$.
The free monad on $R_0'$ then provides an extension operation along ${\dom} \cc \eta$.

It will thus pay off to separately develop the free monad sequence in a potentially non-cocomplete setting, finding a formalism that enables us to express the needed colimits.
For fine grained functoriality of the free monad sequence, we then merely have to require preservation of the postulated colimits.
This development will be undertaken separately in \cref{free-monad-sequence} as we believe it to be of independent interest.

\section{Preliminaries}
\label{preliminaries}

\subsection{Van Kampen colimits}

Let $\cal{E}$ be a category.
We write $\cal{E}^\to$ for its arrow category and $\cal{E}^\to_{\cart}$ for the wide subcategory of $\cal{E}^\to$ with morphisms restricted to pullback squares.
Recall the following notion from~\cite{heindel:van-kampen-universal}.
\notec{What is an earlier reference?}

\begin{definition} \label{van-kampen}
A colimiting cocone $u$ of a diagram $F \co \cal{S} \to \cal{E}$ is called \emph{van Kampen} if:
\begin{enumerate}[label=(\roman*)]
\item it is stable under pullback,
\item for any lift of $F$ through $\cod \co \cal{E}^\to_{\cart} \to \cal{E}$ and compatible lift of $u$ through $\cod \co \cal{E}^\to \to \cal{E}$, we have that $u$ lifts through $\cal{E}^\to_{\cart} \to \cal{E}^\to$.
\end{enumerate}
\end{definition}

If $\cal{E}$ has pullbacks, then a colimit in $\cal{E}$ van Kampen precisely if it is preserved (as a weak 2-limit) by the self-indexing weak 2-functor $\cal{E}^{\op} \to \Cat$.
Our main usage pattern of the van Kampen condition is given by the following statement.

\begin{proposition} \label{van-kampen-lift}
Let $\Xi$ be a pullback-stable class of van Kampen colimiting cocones in $\cal{E}$.
Then $\cod \co \cal{E}^\to_{\cart} \to \cal{E}$ lifts colimits in $\Xi$, and $\cal{E}^\to_{\cart} \to \cal{E}^\to$ preserves these lifted colimits.
\end{proposition}

\begin{proof}
Given a diagram $F \co \cal{S} \to \cal{E}^\to_{\cart}$ and a colimiting cocone $u$ of $\cod \cc F \co \cal{S} \to \cal{S}$ in $\cal{X}$, pullback stability of $\cal{X}$ implies that $\dom \cc F$ has a colimit, hence that the colimit of $F \co \cal{S} \to \cal{E}^\to$ exists.
Using the above consequence~(ii') of the van Kampen condition, it lifts further to a colimit of $F$ in $\cal{E}^\to_{\cart}$.
\end{proof}

\subsection{Categories with a class of maps}

\begin{definition}
The category of \emph{categories with classes of maps} is defined as follows.
An object $(\cal{E}, \cal{M})$ consists of a category $\cal{E}$ with a class $\cal{M}$ of maps.
A map from $(\cal{E}_1, \cal{M}_1)$ to $(\cal{E}_2, \cal{M}_2)$ is a functor $F \co \cal{E}_1 \to \cal{E}_2$ such that $F(\cal{M}_1) \subseteq \cal{M}_2$.
\end{definition}

\begin{definition}
We say that $(\cal{E}, \cal{M})$ \emph{has pullbacks} if pullbacks along maps in $\cal{M}$ exist and $\cal{M}$ is stable under pullback.
\end{definition}

\begin{definition}
Let $\cal{E}$ be a category with a class $\cal{M}$ of maps.
We write $\cal{E}_{\cal{M}}$ for the wide subcategory of $\cal{E}$ with morphisms restricted to $\cal{M}$.
We write $\cal{M}_{\cart}$ for the full subcategory of $\cal{E}_{\cart}$ with objects restricted to $\cal{M}$.
In $\cal{E}^\to$ or one of its subcategories, we write $\overline{\cal{M}}$ for the class of maps $(u, v)$ with $u, v \in \cal{M}$.
\end{definition}

\begin{definition} \label{composition}
Let $\kappa > 0$ be a regular ordinal.
We say that $(\cal{E}, \cal{M})$ \emph{has $\kappa$-compositions} if:
\begin{enumerate}[label=(\roman*)]
\item
$\cal{M}$ is closed under finitary composition;
\item
$\cal{E}_{\cal{M}}$ has and $\cal{E}_{\cal{M}} \to \cal{E}$ preserves $\alpha$-transfinite compositions for any limit ordinal $\alpha \leq \kappa$.
\end{enumerate}
We also use the term $\kappa$-compositions in $(\cal{E}, \cal{M})$ to refer to the class of colimits in~(ii).
We say that $(\cal{E}, \cal{M})$ has \emph{van Kampen $\kappa$-compositions} if they are futhermore van Kampen.
\end{definition}

Assuming condition~(i), condition~(ii) can also be phrased as follows.
Let $\alpha \leq \kappa$ be a limit ordinal and $F \co \alpha \to \cal{E}$ be a cocontinuous diagram with successor maps $F(\beta) \to F(\beta + 1)$ in $\cal{M}$.
Then $F$ admits a colimit, the coprojections $F(\beta) \to \colim F$ are in $\cal{M}$, and for any cocone $(X, u)$ with maps $u_\beta \co F(\beta) \to X$ in $\cal{M}$, the induced map $\colim F \to X$ is in $\cal{M}$ as well.
The situation is illustrated in the below diagram, with the membership in $\cal{M}$ of the dotted arrows being postulated:
\[
\xymatrix{
  F(0)
  \ar[r]^{\in \cal{M}}
  \ar@{.>}[drrr]|(0.33){\in \cal{M}}
  \ar[ddrrr]_{\in \cal{M}}
&
  F(1)
  \ar[r]^{\in \cal{M}}
  \ar@{.>}[drr]|{\in \cal{M}}
  \ar[ddrr]|(0.6){\in \cal{M}}
&
  \ldots
\\&&&
  \colim_F
  \ar@{.>}[d]^{\in \cal{M}}
\\&&&
  X
\rlap{.}}
\]

\begin{corollary} \label{van-kampen-composition}
Let $(\cal{E}, \cal{M})$ have pullbacks and van Kampen $\kappa$-compositions.
Then $(\cal{E}^\to_{\cart}, \overline{\cal{M}})$ has $\kappa$-compositions and all maps in
\[
\xymatrix{
  (\cal{E}^\to_{\cart}, \overline{\cal{M}})
  \ar[r]
  \ar[dr]
&
  (\cal{E}^\to, \overline{\cal{M}})
  \ar[d]
\\&
  (\cal{E}, \cal{M})
}
\]
lift and preserve $\kappa$-compositions.
\end{corollary}

\begin{proof}
By \cref{van-kampen-lift}.
\end{proof}

\begin{definition} \label{pushout}
Let $\cal{E}$ be a category with a class $\cal{M}$ of maps.
We say that $(\cal{E}, \cal{M})$ \emph{has pushouts} if pushouts along maps in $\cal{M}$ exist and $\cal{M}$ is stable under pushout.
We also use the term pushouts in $(\cal{E}, \cal{M})$ to refer to the above class of colimits.
We say that $(\cal{E}, \cal{M})$ has \emph{van Kampen pushouts} if they are furthermore van Kampen.
\end{definition}

In detail, given a span $C \leftarrow A \to B$ with left leg $A \to C$ in $\cal{M}$, the pushout
\[
\xymatrix{
  A
  \ar[r]
  \ar[d]^{\in \cal{M}}
&
  B
  \ar[d]^{\in \cal{M}}
\\
  C
  \ar[r]
&
  D
  \pullback{ul}
}
\]
exists and $B \to D$ is again in $\cal{M}$.

\begin{remark} \label{adhesive}
Let $(\cal{E}, \cal{M})$ have pullbacks and van Kampen pushouts.
Then the maps in $\cal{M}$ are \emph{adhesive} in the sense of~\cite{garner-lack:adhesive}.
In particular, all maps in $\cal{M}$ are regular monomorphisms. 
\end{remark}

In contrast to \cref{van-kampen-composition}, the following statement needs arbitrary pullbacks in $\cal{E}$ to exist.
\notec{Check}

\begin{corollary} \label{van-kampen-pushout}
Let $(\cal{E}, \cal{M})$ have pullbacks and van Kampen $\kappa$-compositions.
Assume that $\cal{E}$ has all pullbacks.
Then $(\cal{E}^\to_{\cart}, \overline{\cal{M}})$ has $\kappa$-compositions and all maps in
\[
\xymatrix{
  (\cal{E}^\to_{\cart}, \overline{\cal{M}})
  \ar[r]
  \ar[dr]
&
  (\cal{E}^\to, \overline{\cal{M}})
  \ar[d]
\\&
  (\cal{E}, \cal{M})
}
\]
lift and preserve $\kappa$-compositions.
\end{corollary}

\begin{proof}
By \cref{van-kampen-lift}.
\end{proof}

\begin{definition} \label{binary-union}
Let $(\cal{E}, \cal{M})$ have pullbacks and pushouts.
We say that $(\cal{E}, \cal{M})$ \emph{has binary unions} if subobjects in $\cal{M}$ are stable under binary union.
\end{definition}

In detail, the definition says the following.
Given maps $B_1 \to D$ and $B_2 \to D$ in $\cal{M}$, consider the pullback
\[
\xymatrix{
  A
  \ar[r]^{\in \cal{M}}
  \ar[d]^{\in \cal{M}}
  \pullback{dr}
&
  B_1
  \ar[d]^{\in \cal{M}}
\\
  B_2
  \ar[r]^{\in \cal{M}}
&
  D
\rlap{.}}
\]
Then the pushout corner map, \ie the induced map from the pushout of $B_1 \leftarrow A \to B_2$ to $D$, is in $\cal{M}$.
The term \emph{union} in this situation is sensible only in case maps in $\cal{M}$ are mono, which in view of \cref{adhesive} is enforced if $(\cal{E}, \cal{M})$ has van Kampen pushouts.
However, we still use the term in the general situation.

\begin{lemma} \label{adequate-pushout}
Let $(\cal{E}, \cal{M})$ have pullbacks and van Kampen pushouts.
Then pushout of maps in $\cal{M}$ lifts to a functor $P \co \cal{M}_{\cart} \times_{\cal{E}} \cal{E}^\to \to \cal{M}_{\cart}$.

If $(\cal{E}, \cal{M})$ has binary unions and $\cal{M}$ is closed under binary composition, then this functor furthermore preserves $\overline{\cal{M}}$: given a map $(u, v)$ in $\cal{M}_{\cart} \times_{\cal{E}} \cal{E}^\to$ with $u, v \in \overline{\cal{M}}$, then $P(u, v) \in \overline{\cal{M}}$.
\end{lemma}

\begin{proof}
The action on objects is clear from the assumptions.
For the action on morphisms, suppose we are given squares
\begin{align*}
\xymatrix{
  A_1
  \ar[r]
  \ar[d]^{\in \cal{M}}
  \pullback{dr}
&
  A_2
  \ar[d]^{\in \cal{M}}
\\
  B_1
  \ar[r]
&
  B_2
\rlap{,}}
&&
\xymatrix{
  A_1
  \ar[r]
  \ar[d]
&
  A_2
  \ar[d]
\\
  C_1
  \ar[r]
&
  C_2
\rlap{,}}
\end{align*}
We need to show that the right face of
\begin{equation} \label{adequate-pushout:0}
\begin{gathered}
\xymatrix{
  A_1
  \ar[rr]
  \ar[dd]^{\in \cal{M}}
  \ar[dr]
  \fancypullback{[dd]}{[dr]}
&&
  C_1
  \ar[dd]|!{[dl];[dr]}{\hole}^(0.7){\in \cal{M}}
  \ar[dr]
\\&
  A_2
  \ar[rr]
  \ar[dd]^(0.3){\in \cal{M}}
&&
  C_2
  \ar[dd]^{\in \cal{M}}
\\
  B_1
  \ar[rr]|!{[ur];[dr]}{\hole}
  \ar[dr]
&&
  D_1
  \ar[dr]
  \fancypullback{[ll]}{[uu]}
\\&
  B_2
  \ar[rr]
&&
  \fancypullback{[ll]}{[uu]}
  D_2
}
\end{gathered}
\end{equation}
is a pullback.
This is a consequence of stability of $\cal{M}$ under pullback and stability under pullback of pushouts along $\cal{M}$ as implied by the van Kampen condition, for example using~\cite[Lemma~2.2]{garner-lack:adhesive}.

For the final assertion, note first that in a square
\begin{equation*} \label{adequate-pushout:1}
\begin{gathered}
\xymatrix{
  X_1
  \ar[r]^{\in \cal{M}}
  \ar[d]^{\in \cal{M}}
  \pullback{dr}
&
  X_2
  \ar[d]^{\in \cal{M}}
\\
  Y_1
  \ar[r]
&
  Y_2
}
\end{gathered}
\end{equation*}
the bottom map is in $\cal{M}$ just when the pushout corner map is in $\cal{M}$.
The forward direction follows from closure of $\cal{M}$ under binary union, the reverse one from closure of $\cal{M}$ under pushout and composition.
Now assume in the situation of~\eqref{adequate-pushout:0} that the maps $A_1 \to A_2$, $B_1 \to B_2$, $C_1 \to C_2$ are in $\cal{M}$.
Our goal is to show that then also $D_1 \to D_2$ is in $\cal{M}$.
By the previous paragraph, the pushout corner map in the left square of~\eqref{adequate-pushout:0} is in $\cal{M}$.
As a pushout of that map, the pushout corner map in the right square of~\eqref{adequate-pushout:0} is then also in $\cal{M}$.
Again by the previous paragraph, it follows that the map $D_1 \to D_2$ is in $\cal{M}$.
\end{proof}

\subsection{Strong categories of functors, adjunctions, monads}

Functoriality of the free algebra and free monad construction will in each case be expressed by means of a functor from a certain category of configurations to strong (non-lax) variants of the categories $\Adju$ and $\Mnd$ of adjunctions and monads, respectively, which we briefly recall below.

\begin{definition} \label{cat-of-functors}
The category $\Fun_s$ of functors with strong morphisms is defined as follows:
\begin{enumerate}[label=(\roman*)]
\item
An object consists of categories $\cal{C}$ and $\cal{D}$ with a functor $F \co \cal{C} \to \cal{D}$.
\item
A morphism from $(\cal{C}, \cal{D}, F)$ to $(\cal{C}', \cal{D}', F')$ consists of functors $U \co \cal{C} \to \cal{C}'$ and $V \co \cal{D} \to \cal{D}'$ with an isomorphism $V F \simeq F' U$.
\end{enumerate}
\end{definition}

\begin{definition} \label{cat-of-adjunctions}
The category $\Adju_s$ of adjunctions with strong morphisms is defined as follows:
\begin{enumerate}[label=(\roman*)]
\item
An object consists of categories $\cal{C}$ and $\cal{D}$ with functors $F \co \cal{C} \to \cal{D}$ and $G \co \cal{D} \to \cal{C}$ admitting natural transformations $\eta \co \Id \to GF$ and $\epsilon \co FG \to \Id$ such that $\epsilon F \cc F \eta = \id$ and $G \epsilon \cc \eta G = \id$.
\item
A morphism from $(\cal{C}, \cal{D}, F, G, \eta, \epsilon)$ to $(\cal{C}', \cal{D}', F', G', \eta', \epsilon)$ consists of functors $U \co \cal{C} \to \cal{C}'$ and $V \co \cal{D} \to \cal{D}'$ with isomorphisms $\alpha \co F' U \simeq V F$ and $\beta \co U G \simeq G' V$ satisfying $\beta F \cc U \eta = G' \alpha \cc \eta' U$ and $V \epsilon \cc \alpha G = \epsilon' V \cc F' \beta$.
\end{enumerate}
\end{definition}

Note that the second equation of part~(ii) of the preceding definition is already implied by the first.
There are functors from $\Adju_s$ to $\Fun_s$ selecting the left and right adjoint, respectively.
Observe that there are lax version of $\Fun_s$ and $\Adju_s$ for which these functors are fully faithful, whereas this is not the case here.
Requiring $\alpha$ to be an isomorphism instead of just a natural transformation is what will encode preservation of free algebras in \cref{wellpointed-free-algebras,pointed-free-algebras}.

\begin{definition} \label{cat-of-monads}
The category $\Mnd_s$ of monads with strong morphisms is defined as follows:
\begin{enumerate}[label=(\roman*)]
\item
An object is a monad $(T, \eta, \mu)$ on a category $\cal{C}$.
\item
A morphism from $(\cal{C}, T, \eta, \mu)$ to $(\cal{C}', T', \eta', \mu')$ is a functor $U \co \cal{C} \to \cal{C}'$ with an isomorphism $\gamma \co U T \simeq T' U$ satisfying $\gamma \cc U \eta = \eta' U$ and $\gamma \cc U \mu = \mu' U \cc T' \gamma \cc \gamma T$.
\end{enumerate}
\end{definition}

There is a functor from $\Mnd_s$ to $\Fun_s$ forgetting the monad structure and the endofunctor aspect.
Finally, there is a functor from $\Adju_s$ to $\Mnd_s$ sending an adjunction to its associated monad.
Similar to the above, observe that $\Mnd_s$ is a wide subcategory of its more common lax version.

The above categories and the below categories of configurations are in fact strict 2-categories and the functors relating them will be strict 2-functors, but we shall not need this here.

\section{Functoriality of free algebras, monads, and monoids}
\label{free-monad-sequence}

The goal of this section is to review the free algebra and free monad constructions of~\cite{kelly:transfinite} on wellpointed and pointed endofunctors (and ultimately the free monoid construction on a pointed object in a monoidal category) without the blanket assumption that all small colimits exist.
Instead, we make the needed colimits explicit, parametrizing them by a class of maps $\cal{M}$ satisfying certain closure properties.
This allows us to apply the constructions in non-cocomplete settings, but the key point is rather the more refined functoriality principles it provides, requiring the translation functor between two settings to only preserve the parametrizing class of maps and associated colimits for free algebras to be preserved and free monads to be related.
We also use the parametrizing class $\cal{M}$ to express the convergence condition in a manner which is sufficient for our purposes, preferring not to introduce the more complicated machinery of well-copowered orthogonal factorization systems and tightness-preserving endofunctors of~\cite{kelly:transfinite}.

We first treat the case of a wellpointed endofunctor and then reduce the case of a pointed endofunctor to it, following faithfully the constructions of~\cite{kelly:transfinite}.
The categories of configurations will involve an ordinal $\kappa$ related to convergence of transfinite sequence constructions.
For simplicity of presentation, we prefer to fix $\kappa$ uniformly for all objects, only noting here that a more flexible treatment is possible.

We first state all the involved definitions, theorems, and corollaries.
The main proofs are offloaded to the last subsection and consist of an analysis of the colimits involved in the construction of~\cite{kelly:transfinite}, though for readability we prefer to give them in whole.

\subsection{Wellpointed endofunctors}

We follow the analysis in~\cite{kelly:transfinite} in first dealing with wellpointed endofunctors.
The statements in here should be seen as technical devices to be used in the general pointed case.

\begin{definition} \label{wellpointed-configuration}
Let $\kappa$ be a regular ordinal.
The (large) category $\ConfWP{\kappa}$ of configurations for the free monad sequence on a wellpointed endofunctor is defined as follows:
\begin{enumerate}[label=(\roman*)]
\item
An object is a tuple $(\cal{E}, \cal{M}, F)$ of a category $\cal{E}$ with a class of maps $\cal{M}$ and a wellpointed endofunctor $(F, \eta)$ on $\cal{E}$ such that:
\begin{enumerate}[label=(\thedefinition.\arabic*)]
\item \label{wellpointed-configurations:composition}
$(\cal{E}, \cal{M})$ has $\kappa$-compositions.
\item \label{wellpointed-configurations:unit}
$\eta$ is valued in $\cal{M}$,
\item \label{wellpointed-configurations:convergence}
$F$ preserves $\kappa$-transfinite compositions of maps in $\cal{M}$.
\end{enumerate}
\item
A morphism from $(\cal{E}_1, \cal{M}_1, F_1)$ to $(\cal{E}_2, \cal{M}_2, F_2)$ is a map $P \co (\cal{E}_1, \cal{M}_1) \to (\cal{E}_1, \cal{M}_2)$ that preserves $\kappa$-compositions and extends to a map of pointed endofunctors from $F_1$ to $F_2$.
\end{enumerate}
\end{definition}

The below theorem asserts that the forgetful functor out of the pointed endofunctor algebra category functorially admits a left adjoint.

\begin{theorem} \label{wellpointed-free-algebras}
The functor $\ConfWP{\kappa} \to \Fun_s$ sending $(\cal{E}, \cal{M}, F)$ to the forgetful functor $\alg(F) \to \cal{E}$ lifts through the restriction $\Adju_s \to \Fun_s$ to the right adjoint.
\end{theorem}

\begin{remark}
There are two aspects to this statement and the one of \cref{pointed-free-algebras}.
The first is to ensure that free algebras exist for any fixed configuration.
They will be constructed explicitly using certain colimits, closely following~\cite{kelly:transfinite}.
This already ensures that the functor $\Conf \to \Fun_s$ from the respective configuration category lifts to $\Adju$, defined as $\Adju_s$ in \cref{cat-of-adjunctions} only with $\alpha \co F' U \to V F$ in the definition of morphisms not required to be invertible: the restriction from $\Adju$ to the right adjoint is fully faithful.

It then remains to show that $\Conf \to \Adju_s$ lifts through the inclusion $\Adju_s \to \Adju$, \ie to show that the $\alpha$-components of the action of $\Conf \to \Adju$ on morphisms are invertible.
This corresponds to showing that the colimits invoked in constructing free algebras are preserved by the mediating functor, which is how morphisms in $\Conf$ are defined.
\end{remark}

\begin{proof}
Let $(\cal{E}, \cal{M}, F) \in \ConfWP{\kappa}$ be a configuration with $(F, \eta)$ a wellpointed endofunctor.
Recall that the forgetful functor $\alg(F) \to \cal{E}$ is a fully faithful inclusion, with $\alg(F)$ consisting of those $X \in \cal{E}$ for which $\eta_X$ is invertible, the algebra map being given by $\eta_X^{-1}$.
Its left adjoint is constructed as the $\kappa$-transfinite composition
\begin{equation} \label{wellpointed-free-algebras:0}
\begin{gathered}
\xymatrix{
  \Id
  \ar[r]^{\eta}
&
  F
  \ar[r]^{\eta F}
&
  F^2
  \ar[r]^{\eta F^2}
&
  \ldots
}
\end{gathered}
\end{equation}
where the successor of $F^\alpha$ is constructed as $\eta F^\alpha \co F^\alpha \to F^{\alpha+1} \defeq F F^\alpha$.
This transfinite composition exists by~\ref{wellpointed-configurations:composition} since the successor maps have components in $\cal{M}$ by~\ref{wellpointed-configurations:unit}.
By~\ref{wellpointed-configurations:convergence}, $F$ preserves this $\kappa$-transfinite composition, making $\eta F^\kappa$ invertible using that $F$ is wellpointed and thus producing a lift $F^\kappa \co \cal{E} \to \alg(F)$.
It is easy to check that $F^\kappa$ is left adjoint to the forgetful functor $\alg(F) \to \cal{E}$.
This finishes the construction of the action on objects of $\ConfWP{\kappa} \to \Adju_s$.

Given a morphism $P \co (\cal{E}_1, \cal{M}_1, F_1) \to (\cal{E}_2, \cal{M}_2, F_2)$, the canonical natural transformation $F_2^\kappa  P \to P F_1^\kappa$ is invertible since $P$ preserves $\kappa$-transfinite compositions of maps in $\cal{M}_1$.
This shows functoriality of $\ConfWP{\kappa} \to \Adju_s$.
\end{proof}

As a consequence, free monads exist functorially.

\begin{theorem} \label{wellpointed-free-monad}
The canonical functor $\ConfWP{\kappa} \to \Cat$ lifts to a functor $\ConfWP{\kappa} \to \Mnd_s$ sending $(\cal{E}, \cal{M}, F)$ to the free and algebraically-free monad $(T, \eta^T, \mu^T)$ on $(F, \eta^F)$.
The following properties are transferred:
\begin{enumerate}[label=(\roman*)]
\item \label{wellpointed-free-monad:colimit}
Let $\Xi$ be a class of colimits in $\cal{E}$.
If $F$ preserves $\Xi$, then so does $T$.
\item \label{wellpointed-free-monad:unit}
$\eta^T$ is valued in $\cal{M}$.
\item \label{wellpointed-free-monad:functor}
Let $(\cal{C}, \cal{N})$ have $\kappa$-compositions and let $(\cal{E}, \cal{M}) \to (\cal{C}, \cal{N})$ preserve $\kappa$-compositions.
If $F$ preserves maps mapping to $\cal{N}$, then so does $T$.
\item \label{wellpointed-free-monad:cartesian}
Let $(\cal{C}, \cal{N})$ have pullbacks and van Kampen $\kappa$-compositions and let $(\cal{E}, \cal{M}) \to (\cal{C}, \cal{N})$ preserve $\kappa$-compositions.
Let $\cal{S}$ be a class of maps in $\cal{E}$ preserved by $F$.
If naturality squares of $\eta^F$ on $\cal{S}$ get mapped to pullbacks in $\cal{C}$, then the same holds for $\eta^T$.
\end{enumerate}
\end{theorem}

The generalized setup of some of the above clauses is necessitated by how the reduction to the wellpointed case of the corresponding clauses in the pointed case is implemented.

\begin{proof}
The main statement is an immediate consequence of \cref{wellpointed-free-algebras}, noting that the free and algebraically-free monad $T$ on a wellpointed endofunctor $F$ is given by the monad of the free algebra adjunction.
Assertion~\ref{wellpointed-free-monad:colimit} holds by commutativity of colimits.
For assertion~\ref{wellpointed-free-monad:unit}, note that the unit of $T$ is given by the $\kappa$-transfinite composition~\eqref{wellpointed-free-algebras:0} and hence valued in $\cal{M}$ by~\ref{wellpointed-configurations:composition}.
Assertion~\ref{wellpointed-free-monad:functor} follows from the assumptions and~\ref{wellpointed-configurations:composition} by iterative application of $F$ preserving maps mapping to $\cal{N}$.
For assertion~\ref{wellpointed-free-monad:cartesian}, proceed as in~\ref{wellpointed-free-monad:functor} and~\ref{wellpointed-free-monad:unit}, using the van Kampen condition at limit ordinal steps.
\end{proof}

\subsection{Statements for pointed endofunctors}

Due to their lengths, the proofs of the statements in this section are given at the end of this section.

\begin{definition} \label{pointed-configurations}
Let $\kappa > 0$ be a regular ordinal.
The (large) category $\ConfP{\kappa}$ of configurations for the free monad sequence on a pointed endofunctor is defined as follows.
\begin{enumerate}[label=(\roman*)]
\item
An object is a tuple $(\cal{E}, \cal{M}, F)$ of a category $\cal{E}$ with a class of maps $\cal{M}$ and a pointed endofunctor $(F, \eta)$ on $\cal{E}$ such that:
\begin{enumerate}[label=(\thedefinition.\arabic*)]
\item \label{pointed-configurations:composition}
$(\cal{E}, \cal{M})$ has $\kappa$-compositions.
\item \label{pointed-configurations:pushout}
$(\cal{E}, \cal{M})$ has pushouts.
\item \label{pointed-configurations:unit}
$\eta$ is valued in $\cal{M}$,
\item \label{pointed-configurations:leibniz-application}
$\cal{M}$ is closed under Leibniz application of $\eta$,
\item \label{pointed-configurations:functor}
$F$ preserves $\cal{M}$,
\item \label{pointed-configurations:convergence}
$F$ preserves $\kappa$-transfinite compositions of maps in $\cal{M}$.
\end{enumerate}
\item
A morphism from $(\cal{E}_1, \cal{M}_1, F_1)$ to $(\cal{E}_2, \cal{M}_2, \cal{F}_2)$ is a map $P \co (\cal{E}_1, \cal{M}_1) \to (\cal{E}_2, \cal{M}_2)$ that preserves $\kappa$-compositions and pushouts and extends to a map of pointed endofunctors from $F_1$ to $F_2$.
\end{enumerate}
\end{definition}

\begin{remark} \label{pointed-configuration-simplification}
Let $(\cal{E}, \cal{M}, F) \in \ConfP{\kappa}$.
If $\cal{M}$ consists only of monos, $\cal{E}$ has pullbacks along maps in $\cal{M}$, and crucially the unit of $F$ is cartesian, condition~\ref{pointed-configurations:leibniz-application} is satisfied if $\cal{M}$ is closed under binary union.
\end{remark}

The below statements are similar to the wellpointed case.

\begin{theorem} \label{pointed-free-algebras}
The functor $\ConfP{\kappa} \to \Fun_s$ sending $(\cal{E}, \cal{M}, F)$ to the forgetful functor $\alg(F) \to \cal{E}$ lifts through the restriction $\Adju_s \to \Fun_s$ to the right adjoint.
\end{theorem}

\begin{theorem} \label{pointed-free-monad}
The canonical functor $\ConfP{\kappa} \to \Cat$ lifts to a functor $\ConfP{\kappa} \to \Mnd_s$ sending $(\cal{E}, \cal{M}, F)$ to the free and algebraically-free monad $(T, \eta^T, \mu^T)$ on $F$.
It furthermore shares the following properties with $F$:
\begin{enumerate}[label=(\roman*)]
\item \label{pointed-free-monad:colimit}
Let $\Xi$ be a class of colimits in $\cal{E}$.
If $F$ preserves $\Xi$, then so does $T$.
\item \label{pointed-free-monad:unit}
$\eta^T$ is valued in $\cal{M}$,
\item \label{pointed-free-monad:functor}
$T$ preserves $\cal{M}$,
\item \label{pointed-free-monad:cartesian}
Assume $(\cal{E}, \cal{M})$ has pullbacks and van Kampen pushouts.
If $\eta^F$ is cartesian, then so is $\eta^F$.
\end{enumerate}
\end{theorem}

\subsection{Pointed objects in a monoidal category}

Let $\MonCat$ denote the category of monoidal categories.
We default to strong (non-lax) monoidal functors.

\begin{definition} \label{pointed-monoid-configurations}
Let $\kappa > 0$ be a regular ordinal.
The (large) category $\ConfPM{\kappa}$ of configurations for the free monoid sequence on a pointed object is defined as follows:
\begin{enumerate}[label=(\roman*)]
\item
An object is a tuple $(\cal{E}, \cal{M}, T)$ of a monoidal category $(\cal{E}, I, \otimes)$ with a class of maps $\cal{M}$ and a pointed object $(T, \eta)$ in $\cal{E}$ such that $(\cal{E}, \cal{M}, T \otimes (-))$ is an object of $\ConfP{\kappa}$ and $(-) \otimes X$ is an endomorphism on it for in $X \in \cal{E}$.
In detail, we require:
\begin{enumerate}[label=(\thedefinition.\arabic*)]
\item \label{pointed-monoid-configurations:composition}
$(\cal{E}, \cal{M})$ has $\kappa$-compositions.
\item \label{pointed-monoid-configurations:pushout}
$(\cal{E}, \cal{M})$ has pushouts,
\item \label{pointed-monoid-configurations:unit}
$\eta$ is valued in $\cal{M}$,
\item \label{pointed-monoid-configurations:leibniz-application}
$\cal{M}$ is closed under Leibniz monoidal product $\eta \hatotimes (-)$,
\item \label{pointed-monoid-configurations:functor}
$T \otimes (-)$ preserves $\cal{M}$.
\item \label{pointed-monoid-configurations:convergence}
$T \otimes (-)$ preserves $\kappa$-transfinite compositions of maps in $\cal{M}$.
\item \label{pointed-monoid-configurations:right-multiplication}
for any $X \in \cal{E}$, right multiplication $(-) \otimes X$ lifts to an endomorphism on $(\cal{E}, \cal{M})$ that preserves $\kappa$-compositions and pushouts.
\end{enumerate}
\item
A morphism from $(\cal{E}_1, \cal{M}_1, T_1)$ to $(\cal{E}_2, \cal{M}_2, T_2)$ is a map $P \co (\cal{E}_1, \cal{M}_1) \to (\cal{E}_2, \cal{M}_2)$ that preserves $\kappa$-compositions and pushouts and extends to a monoidal functor such that $P(T_1) \simeq T_2$ as pointed objects.
\end{enumerate}
\end{definition}

Let $\Mon$ denote the category of monoidal categories with a monoid.
We have a canonical projection $\Mon \to \MonCat$.

\begin{theorem} \label{pointed-free-monoid}
The canonical functor $\ConfPM{\kappa} \to \MonCat$ lifts to a functor $\ConfP{\kappa} \to \Mon$ sending $(\cal{E}, \cal{M}, T)$ to the free and algebraically monoid on $T$.
\end{theorem}

\begin{proof}
This is by \cref{pointed-free-algebras,pointed-free-monad} as detailed in \cite[Section~23]{kelly:transfinite}.
For $(\cal{E}, \cal{M}, T) \in \ConfPM{\kappa}$, the free monoid on $T$ is given by the free monad on $T \otimes (-)$ applied to the unit of $\cal{E}$.
\end{proof}

\begin{remark}
As described in \cite{kelly:transfinite}, we may discuss free monads and free monoids in terms of each other:
\begin{itemize}
\item 
Functorially $(\cal{E}, \cal{M}, T) \in \ConfPM{\kappa}$, we have $(\cal{E}, \cal{M}, T \otimes (-)) \in \ConfP{\kappa}$ and furthermore stable under $(-) \otimes X$ for $X \in \cal{E}$.
\item
Given in $(\cal{E}, \cal{M}, F) \in \ConfP{\kappa}$, we have $([\cal{E}, \cal{E}], \cal{M}', F \cc (-)) \in \ConfPM{\kappa}$ where $\cal{M}'$ is the class of natural transformations whose components lie in $\cal{M}$.
\end{itemize}
However, the second point does not translate to the level of functors: a morphism in $\ConfP{\kappa}$ does not give rise to a morphism in $\ConfPM{\kappa}$.
This means that the free monoid sequence, as a functorial construction, does not generalize the free monad sequence.
\end{remark}

\subsection{Proofs for pointed endofunctors}

\begin{proof}[Proof of \cref{pointed-free-algebras}]
Let $(\cal{E}, \cal{M}, F) \in \ConfP{\kappa}$ be a configuration with $(F, \eta^F)$ a pointed endofunctor.
We want to construct a left adjoint to the forgetful functor $\alg(F) \to \cal{E}$.
The strategy is to reduce to the wellpointed case.
For this, we will define a wellpointed configuration $(\cal{E}', \cal{M}', F') \in \ConfWP{\kappa}$, deriving the connection to the original pointed configuration as we go along.
We identify $\cal{M}$ with the full subcategory of $\cal{E}^\to$ it generates.
Without loss of generality, we can assume it is replete.

Define $\cal{E}'$ via the following pullback, which is also a weak 2-pullback since $\cal{M}$ is replete:
\begin{equation} \label{pointed-free-algebras:0}
\begin{gathered}
\xymatrix{
  \cal{E}'
  \pullback{dr}
  \ar[r]^{K'}
  \ar@{}[d]^(0.15){}="a"
  \ar@{_{(}->} "a";[d]
&
  \cal{M}
  \ar@{}[d]^(0.15){}="a"
  \ar@{_{(}->} "a";[d]
\\
  F \downarrow \cal{E}
  \ar[r]^{K}
&
  \cal{E}^\to
\rlap{.}}
\end{gathered}
\end{equation}
Here, we denote $K \co F \downarrow \cal{E} \to \cal{E}^\to$ the functor that sends $(A, B, f)$ with $A, B \in \cal{E}$ and $f \co F A \to B$ to $f \cc \eta_A^F$.
Explicitly, we can describe $\cal{E}'$ as the full subcategory of $F \downarrow \cal{E}$ of objects $(A, B, f)$ such that the composite $f \cc \eta^F_A \co A \to B$ is in $\cal{M}$.

The forgetful functor $\alg(F) \to \cal{E}$ has an evident lift through the projection $F \downarrow \cal{E} \to \cal{E}$ sending $(A, B, f)$ to $A$.
It further lifts through $\cal{E}' \to F \downarrow \cal{E}$ since identities are contained in $\cal{M}$ by~\ref{pointed-configurations:composition}:
\begin{equation} \label{pointed-free-algebras:0.75}
\begin{gathered}
\xymatrix{
  \alg(F)
  \ar[r]
&
  \cal{E}'
  \ar[r]
&
  F \downarrow \cal{E}
  \ar[r]
&
  \cal{E}
\rlap{.}}
\end{gathered}
\end{equation}
Note that $\cal{E}' \to \cal{E}$ has a left adjoint sending $X$ to $(X, F X, \id_{F X})$.
Thus, the problem of constructing a left adjoint to $\alg(F) \to \cal{E}$ reduces to constructing a left adjoint to $\alg(F) \to \cal{E}'$.

Let $\cal{E}^{\iso}$ denote the category of isomorphisms in $\cal{E}$.
The outer square below forms a weak 2-pullback:
\[
\xymatrix{
  \alg(F)
  \ar[r]
  \ar[d]
  \pullback{dr}
&
  \cal{E}^\simeq
  \ar[d]
\\
  \cal{E}'
  \ar[r]^{K'}
  \ar[d]
  \pullback{dr}
&
  \cal{M}
  \ar[d]
\\
  F \downarrow \cal{E}
  \ar[r]^{K}
&
  \cal{E}^\to
\rlap{.}}
\]
The lower square is the weak 2-pullback~\eqref{pointed-free-algebras:0}, so the upper square is one as well.
Note that $K'$ has a left adjoint $J'$ sending an arrow $f \co A \to B$ to the triple $(A, C, g)$ defined by the pushout
\begin{equation} \label{pointed-free-algebras:0.5}
\begin{gathered}
\xymatrix{
  A
  \ar[r]^{\eta^F_A}
  \ar[d]_{f}
&
  F A
  \ar[d]^{g}
\\
  B
  \ar[r]
&
  \pullback{ul}
  C
\rlap{,}}
\end{gathered}
\end{equation}
which exists by~\ref{pointed-configurations:pushout} since $\eta^F_A \in \cal{M}$ by~\ref{pointed-configurations:unit}.
Also note that $\cal{E}^\simeq \to \cal{M}$ is the category of algebras for the wellpointed endofunctor $(U, \eta^U)$ on $\cal{M}$ that sends $m \co A \to B$ to $\id_B$.
Summarizing, we have the following situation:
\begin{equation} \label{pointed-free-algebras:1}
\begin{gathered}
\xymatrix{
  \alg(F)
  \ar@/^/[r]
  \ar[d]
  \pullback{dr}
&
  \alg(U)
  \ar[d]
\\
  \cal{E'}
  \ar@/^/[r]^-{K'}
&
  \cal{M}
  \ar@/^/[l]^-{J'}
  \ar@{}@<-0.25em>[l]|-{\dir{_|_}}
\rlap{.}}
\end{gathered}
\end{equation}

We now apply \cref{adjunction-transfer-pointed} below to the adjunction $J' \dashv K'$ and the wellpointed endofunctor $U$.
This will define a wellpointed endofunctor ($F', \eta^{F'})$ on $\cal{E}'$ via the following pushout for $(A, B, f) \in \cal{E}'$ where we write $(X, Y, k) \defeq F'(A, B, f)$:
\begin{equation} \label{pointed-free-algebras:2}
\begin{gathered}
\xymatrix@C+1em{
  (A, B +_A F A, \inr)
  \ar[r]^-{J' \eta^U K'}
  \ar[d]
&
  (B, F B, \id_{F B})
  \ar[d]
\\
  (A, B, f)
  \ar[r]^{\eta^{F'}}
&
  (X, Y, k)
  \pullback{ul}
\rlap{.}}
\end{gathered}
\end{equation}
We will show first that the pushout exists in $F \downarrow \cal{E}$ and then that it lifts to $\cal{E}'$.
The pushout $X$ of the domain components is trivially $B$ and furthermore preserved by $F$ as it is absolute.
Hence $Y$ is simply computed as a pushout of the codomain components:
\begin{equation} \label{pointed-free-algebras:3}
\begin{gathered}
\xymatrix@C+1em{
  B +_A FA
  \ar[r]^-{\eta^F \hatotimes f \eta^F_A}
  \ar[d]
&
  FB
  \ar[d]
\\
  B
  \ar[r]^-{h}
&
  \pullback{ul}
  Y
\rlap{.}}
\end{gathered}
\end{equation}
The top map is the Leibniz application of $\eta$ to $f \eta^F_A$ and thus in $\cal{M}$ by~\ref{pointed-configurations:leibniz-application} since $f \eta^F_A$ is in $\cal{M}$, hence the pushout exists and the bottom map $h$ is in $\cal{M}$ by~\ref{pointed-configurations:pushout}.
The map $k \co FY \to X$ is given by functoriality of pushouts in the cube
\[
\begin{tikzcd}[column sep=4em]
  FA
  \ar[rr, "F(f \cc \eta_A)"]
  \ar[dd, "\tau_1"]
  \ar[dr, equals]
&&
  FB
  \ar[dd, equals]
  \ar[dr, equals]
\\&
  FA
  \ar[rr, crossing over, "F(f \cc \eta_A)", near start]
&&
  FB
  \ar[dd, "k"]
\\
  B +_A FA
  \ar[rr, "\eta \hatotimes (f \cc \eta_A)", near start]
  \ar[dr, "\bracks{\id, f}"']
&&
  FB
  \ar[dr, "k"]
\\&
  B
  \ar[from=uu, crossing over, "f", near end]
  \ar[rr, "h"]
&&
  Y \rlap{,}
\end{tikzcd}
\]
whose top and bottom faces are pushouts.
Note that
\begin{align*}
  h
&=
  h \cc \bracks{\id, f} \cc \tau_1
\\&=
  k \cc (\eta \hatotimes (f \cc \eta_A)) \cc \tau_1
\\&=
  k \cc \eta_B \rlap{.}
\end{align*}
This verifies that $(X, Y, k) \in \cal{E}'$ since $k \cc \eta^F_X = h$ is in $\cal{M}$.

We continue the application of \cref{adjunction-transfer-pointed}.
By comparing the 2-pullbacks~\eqref{adjunction-transfer-pointed:0} and~\eqref{pointed-free-algebras:1}, we see that the categories $\alg(F)$ and $\alg(F')$ are equivalent over $\cal{E}'$.
Instead of constructing a left adjoint to $\alg(F) \to \cal{E}'$, we construct one to $\alg(F') \to \cal{E}'$, exploiting that now $F'$ is wellpointed so that we may apply \cref{wellpointed-free-algebras}.

We let $\cal{M}'$ be the class of arrows $(u, v)$ of $\cal{E}'$ with $u, v \in \cal{M}$.
Let us show that $\cal{E}'$ has $\alpha$-transfinite composition of maps in $\cal{M}'$ for any limit ordinal $\alpha \leq \kappa$.
Consider a diagram $(U, V, m) \co \alpha \to \cal{E}'$ with functors $U, V \co \alpha \to \cal{E}$ with successor maps in $\cal{M}$ and a natural transformation $m \co FU \to V$ such that $m \cc \eta^F U$ has components in $\cal{M}$.
We will show that the colimit $(X, Y, k) \defeq \colim_\alpha (U, V, m)$ exists in $F \downarrow \cal{E}$ and lifts to $\cal{E}'$.
By~\ref{pointed-configurations:functor} and~\ref{pointed-configurations:composition}, the colimits of $U$, $V$, $F U$ exist and have coprojections in $\cal{M}$.
We have $X = \colim_\alpha U$ while $Y$ and $k$ are computed via the following pushout:\begin{equation} \label{pointed-free-algebras:4}
\begin{gathered}
\xymatrix{
  \colim_\alpha U
  \ar[r]
  \ar[dr]
&
  \colim_\alpha FU
  \ar[r]
  \ar[d]
&
  \colim_\alpha V
  \ar[d]
\\&
  F \colim_\alpha U
  \ar[r]^-{k}
&
  \pullback{ul}
  Y
\rlap{.}}
\end{gathered}
\end{equation}
Again by~\ref{pointed-configurations:functor} and~\ref{pointed-configurations:composition}, the middle vertical map is in $\cal{M}$.
By~\ref{pointed-configurations:pushout}, the pushout exists and the right map is in $\cal{M}$.
Thus the colimit in $F \downarrow \cal{E}$ exists.
By~\ref{pointed-configurations:composition}, the composite top row is in $\cal{M}$.
The composite $X \to FX \to Y$ factors as a sequence $\colim_\alpha U \to \colim_\alpha V \to Y$ of maps in $\cal{M}$, so is itself in $\cal{M}$ by~\ref{pointed-configurations:composition}.
Thus the colimit lifts to $\cal{E}'$.

Let us verify that the component of the unit $\eta^{F'}$ on $(A, B, f) \in \cal{E}'$ depicted in the bottom row~\eqref{pointed-free-algebras:2} is in $\cal{M}'$.
The domain part was simply $f \cc \eta^F_A A \co A \to B$, which is in $\cal{E}$ by definition of $\cal{E}'$.
The codomain part is the bottom row of~\eqref{pointed-free-algebras:3}, which as we saw there is in $\cal{E}$.

It remains to check that $F'$ preserves $\kappa$-transfinite compositions of maps in $\cal{M}'$.
This is a consequence of~\ref{pointed-configurations:convergence}.
This establishes $(\cal{E}', \cal{M}', F')$ as an object of $\ConfWP{\kappa}$, allowing us to apply the object part of \cref{wellpointed-free-algebras} to finish the construction of the action on objects of $\ConfP{\kappa} \to \Adju_s$.

\addvspace{0.3cm}

Let us now examine functoriality of $\ConfP{\kappa} \to \Adju_s$.
Given a map $P \co (\cal{E}_1, \cal{M}_1, F_1) \to (\cal{E}_2, \cal{M}_2, F_2)$ of pointed configurations, we need to show that $P$ preserves free algebras.
Following the structure of the construction for the object part and reusing its notation annotated with appropriate subscripts, we will show that $P$ induces a morphism $P' \co (\cal{E}_1', \cal{M}_1', F_1') \to (\cal{E}_2', \cal{M}_2', F_2')$ between the associated wellpointed configurations, then apply the functoriality part of \cref{wellpointed-free-algebras}, and finish by relating preservation of free algebras for the associated wellpointed endofunctors to the original pointed endofunctors.

To see that $P$ lifts to a functor $P' \co \cal{E}_1' \to \cal{E}_2'$, we examine the pullback square~\eqref{pointed-free-algebras:0}.
Since $P$ is a map of pointed endofunctors, it lifts to a map of functors from $K_1$ to $K_2$.
Observe that $P$ lifts to a functor $\cal{M}_1 \to \cal{M}_2$ by assumption.
Together, this induces the lift $P' \co \cal{E}_1' \to \cal{E}_2'$.
Note that $P'$ maps $\cal{M}_1'$ to $\cal{M}_2'$ since $P$ maps $\cal{M}_1$ to $\cal{M}_2$.

To see that $P'$ preserves $\alpha$-transfinite compositions of maps in $\cal{M}_1$ for any limit ordinal $\alpha \leq \kappa$, recall the construction of $\alpha$-transfinite compositions in $\cal{E}_1'$ and $\cal{E}_2'$ as depicted in~\eqref{pointed-free-algebras:4} and note that $P$ preserves all of the involved colimits, namely pushouts and $\alpha$-transfinite compositions of maps in $\cal{M}_1$.

To see that $P'$ extends to a map of pointed endofunctors from $F_1'$ to $F_2'$, recall the construction of $F_1'$ and $F_2'$ via \cref{adjunction-transfer-pointed} as explicated in~\eqref{pointed-free-algebras:2}.
Note that $P$ preserves the pushout~\eqref{pointed-free-algebras:0.5} defining $J_1'$ since one of its legs is in $\cal{M}$,.
It follows that $P$ lifts to a map of adjunctions from $(J_1', K_1')$ to $(J_2', K_2')$
Since it also trivially lifts to a map of pointed endofunctors from $U_1$ to $U_2$, functoriality of \cref{adjunction-transfer-pointed} yields the desired map of pointed endofunctors from $F_1'$ to $F_2'$.

This shows that $P'$ is a morphism relating the wellpointed configurations $(\cal{E}_1', \cal{M}_1', F_1')$ and $(\cal{E}_2', \cal{M}_2', F_2')$.
Using \cref{wellpointed-free-algebras}, we have that $P'$ lifts to a map of adjunctions from the free algebra adjunction for $F_1'$ to the one for $F_2'$.

Using \cref{wellpointed-free-algebras}, it follows that $P'$ lifts to a map of adjunctions where the right adjoints are the forgetful functors $\alg(F_1') \to \cal{E}_1$ and $\alg(F_2') \to \cal{E}_2'$, or equivalently the functors $\alg(F_1) \to \cal{E}_1$ and $\alg(F_2) \to \cal{E}_2'$ from~\eqref{pointed-free-algebras:0.75}.
Note that $P$ trivially lifts to a map of adjunctions where the right adjoints are the functors $\cal{E}_1' \to \cal{E}_1$ and $\cal{E}_2' \to \cal{E}_2$ from~\eqref{pointed-free-algebras:0.75}.
Taken together, this implies that $P$ lifts to a map of adjunctions from the free algebra adjunction for $F_1$ to the one for $F_2$.
\end{proof}

\begin{proof}[Proof of \cref{pointed-free-monad}]
The main statement is an immediate consequence of \cref{pointed-free-algebras}, noting that the free and algebraically-free monad $T$ on a wellpointed endofunctor $F$ is given by the monad of the free algebra adjunction.
To verify the assertions~\ref{pointed-free-monad:colimit} to~\ref{pointed-free-monad:cartesian}, recall from~\eqref{pointed-free-algebras:0.75} that the free algebra adjunction for $F$ was up to equivalence constructed as a composite of adjunctions as follows:
\[
\xymatrix@C+1em{
  \alg(F')
  \ar@/_/[r]_-{R_2}
  \ar@{}[r]|-{\bot}
&
  \cal{E}'
  \ar@/_/[l]_-{L_2}
  \ar@/_/[r]_-{R_1}
  \ar@{}[r]|-{\bot}
&
  \cal{E}
  \ar@/_/[l]_-{L_1}
\rlap{.}}
\]
Here, recall the definition~\eqref{pointed-free-algebras:0.5} of $\cal{E}'$ and the wellpointed endofunctor $F'$ defined in~\eqref{pointed-free-algebras:3}.
Note that $L_1$ maps $\cal{M}$ to $\cal{M}'$ by~\ref{pointed-configurations:functor} and $R_1$ maps $\cal{M}'$ to $\cal{M}$, and that the unit of $L_1 \dashv R_1$ is an identity.

Assertion~\ref{pointed-free-monad:colimit} follows from commutativity of colimits and part~\ref{wellpointed-free-monad:colimit} of \cref{wellpointed-free-monad}.
\notec{Check}

The unit of $L_2 \dashv R_2$ is valued in $\cal{M}'$ by part~\ref{wellpointed-free-monad:unit} of \cref{wellpointed-free-monad}.
It follows that the unit of the composite adjunction is valued in $\cal{M}$, proving~\ref{wellpointed-free-monad:unit}.

For assertion~\ref{pointed-free-monad:functor}, we apply part~\ref{wellpointed-free-monad:functor} of \cref{wellpointed-free-monad} to the $\kappa$-composition preserving map $R_1 \co (\cal{E}', \cal{M}') \to (\cal{E}, \cal{M})$.
Since $F'$ preserves maps mapping to $\cal{M}$ via $R_1$,\footnote{Note that in general $F'$ will not preserve $\cal{M}'$.} so does the free monad on $F'$, given by $R_2 L_2$.
It follows that the free monad on $F$, given by $R_1 R_2 L_2 L_1$, preserves $\cal{M}$.

For assertion~\ref{pointed-free-monad:cartesian}, we apply part~\ref{wellpointed-free-monad:cartesian} of \cref{wellpointed-free-monad}, again to the $\kappa$-composition preserving map $R_1 \co (\cal{E}', \cal{M}') \to (\cal{E}, \cal{M})$.
We let $\cal{S}$ be the class of maps $(u, v) \co (A_1, B_1, f_1) \to (A_2, B_2, f_2)$ in $\cal{E}'$ such that
\[
\xymatrix{
  F A_1
  \ar[r]^{F u}
  \ar[d]^{f_1}
&
  F A_2
  \ar[d]^{f_2}
\\
  B_1
  \ar[r]^{v}
&
  B_2
}
\]
is a pullback square.
Note that $L_1$ maps any map to $\cal{S}$ as in this case the vertical maps are identities.
Note also that $R_1$ maps the naturality square of $\eta^{F'}$ on $(u, v)$ to the vertical pasting of the naturality square of $\eta^F$ on $u$ with the above square, which hence will be a pullback if $\eta^F$ is cartesian and $(u, v) \in \cal{S}$.
Note finally that $F'$ preserves $\cal{S}$ using van Kampen properties of pushouts in $(\cal{E}, \cal{M})$, the core part being an application of \cref{adequate-pushout}.
Combining everything, it follows that $\eta^T$ is cartesian.
\end{proof}

\begin{lemma} \label{adjunction-transfer-pointed}
Consider an adjunction
\[
\xymatrix{
  \cal{C}
  \ar@<0.5em>[r]^{L}
  \ar@{}@<-0.24em>[r]|{\dir{_|_}}
&
  \cal{D}
  \ar@<0.5em>[l]^{R}
\rlap{.}}
\]
Let $(U, \eta^U)$ be a pointed endofunctor on $\cal{C}$.
Assume that the following pointwise pushout in $[\cal{D}, \cal{D}]$ exists and is computed pointwise:
\[
\xymatrix@C+1em{
  L R
  \ar[r]^{L \eta^U R}
  \ar[d]
&
  L U R
  \ar[d]
\\
  \Id
  \ar[r]^{\eta^V}
&
  V
  \pullback{ul}
\rlap{.}}
\]
This defines a pointed endofunctor $(V, \eta^V)$ on $\cal{D}$.
The categories of algebras for $U$ and $V$ are related by a pullback, which is also a weak 2-pullback:
\begin{equation} \label{adjunction-transfer-pointed:0}
\begin{gathered}
\xymatrix{
  \alg(V)
  \ar[r]
  \ar[d]
  \pullback{dr}
&
  \alg(U)
  \ar[d]
\\
  \cal{D}
  \ar[r]^{R}
&
  \cal{C}
\rlap{.}}
\end{gathered}
\end{equation}
If $U$ is wellpointed, then so is $V$.
\end{lemma}

\begin{proof}
As in~\cite{kelly:transfinite}.
\end{proof}

\section{Reviewing the algebraic small object argument}

The purpose of this section is to review the algebraic small object argument in a certain non-cocomplete setting.
As such, it contains mostly of a review of \cite{garner:small-object-argument}.
The setting has been chosen minimalistictly so that the assumptions needed in the next section to lift am extension operations from the generating category of arrows to the domain part of the category of coalgebras are satisfied.

Recall the notation from \cref{preliminaries}.
Let $\cal{E}$ be a category with a designated class $\cal{M}$ of morphisms.
Let $\kappa > 0$ be a regular ordinal.

\begin{definition} \label{adequate}
We say that $(\cal{E}, \cal{M}, \kappa)$ is \emph{adequate} if the following hold:
\begin{enumerate}[label={(\thedefinition\alph*)}]
\item \label{adequate:pullback}
$(\cal{E}, \cal{M})$ has pullbacks.
\item \label{adequate:composition}
$(\cal{E}, \cal{M})$ has van Kampen $\kappa$-compositions.
\item \label{adequate:pushout}
$(\cal{E}, \cal{M})$ has van Kampen pushouts.
\item \label{adequate:binary-union}
$(\cal{E}, \cal{M})$ has binary unions.
\end{enumerate}
\end{definition}

Note that condition~\ref{adequate:pullback} implies that the class $\cal{M}$ is replete, \ie given $f, g \in \cal{E}^\to$ with $f \simeq g$, then $f \in \cal{M}$ implies $g \in \cal{M}$.
Recall from \cref{adhesive} that all maps in $\cal{M}$ are regular monomorphisms.
This explains the terminology of binary union of subobjects in~\ref{adequate:binary-union}.

Let now $u \co \cal{I} \to \cal{E}^\to$ be a small category of arrows in $\cal{E}$.

\begin{definition} \label{adequate-u}
We say that $(\cal{E}, \cal{M}, \kappa, \cal{I}, u)$ is \emph{adequate} if the following hold:
\begin{enumerate}[label=(\thedefinition\alph*)]
\item \label{adequate-u:compact}
The functor $u$ is valued in arrows with $\kappa$-compact domain.
\item \label{adequate-u:lift}
The functor $u$ lifts through the inclusion $\cal{M}_{\cart} \to \cal{E}^\to$.
\item \label{adequate-u:colimits}
We have colimits of shape $\cal{I}/f$ in $\cal{M}_{\cart}$ for $f \in \cal{E}^\to$ and they are preserved by $\cal{M}_{\cart} \to \cal{E}^\to$.
\item \label{adequate-u:preservation}
The left Kan extension of $u$ along itself, guaranteed to exist by the previous condition, preserves $\overline{\cal{M}}$.
\item \label{adequate-u:van-kampen}
The colimit of the pointwise formula for the left Kan extension of $u$ along itself is van Kampen.
\end{enumerate}
\end{definition}

\notec{Replace~\ref{adequate-u:preservation} and~\ref{adequate-u:van-kampen} by better conditions.}

The construction~\cite{garner:small-object-argument} of the algebraic weak factorization system free on $\cal{I}$ with $\cal{E}$ cocomplete proceeds via stepwise reflection along the following sequence of fully faithful functors:
\[
\xymatrix{
  \AWFS(\cal{E})
  \ar[r]^{G_1}
&
  \LAWFS(\cal{E})
  \ar[r]^{G_2}
&
  \Cmd(\cal{E}^\to)
  \ar[r]^{G_3}
&
  \Cat/\cal{E}^\to
\rlap{.}}
\]
Let us use this opportunity to introduce some terminology, mostly following~\cite{garner:small-object-argument}: $\Mnd(\cal{E}^\to)$, $\Cmd(\cal{E}^\to)$, $\PtdEndo(\cal{E}^\to)$, $\CoptdEndo(\cal{E}^\to)$ denote the categories of monads, comonads, and (co)pointed endofunctors on $\cal{E}^\to$, respectively; $\AWFS(\cal{E})$ denotes the category of algebraic weak factorization systems on $\cal{E}$; $\LAWFS(\cal{E})$ denotes the category of left halves of algebraic weak factorization systems, or equivalently the category of domain-preserving comonads on $\cal{E}^\to$, \ie comonads over the identity comonad with respect to the codomain fibration; similarly, we write $\RAWFS(\cal{E})$ for the category of right halves of algebraic weak factorization systems, equivalently the category of codomain-preserving monads.

Even in a non-complete setting, this construction still works without modification for adequate $(\cal{E}, \cal{M}, \kappa, \cal{I}, u)$.
This is summarized in the following three propositions.

\begin{proposition} \label{reflecting-G3}
Let $(\cal{E}, \cal{M}, \kappa, \cal{I}, u)$ be adequate.
Then $(\cal{I}, u) \in \Cat/\cal{E}^\to$ has a reflection $U$ along the forgetful functor $\Cmd(\cal{E}^\to) \to \Cat/\cal{E}^\to$ and $U$ lifts through the inclusion $\cal{M}_{\cart} \to \cal{E}^\to$.
Furthermore $U$ preserves $\overline{\cal{M}}$ and $\kappa$-transfinite compositions of maps in $\overline{\cal{M}}$.
\end{proposition}

\begin{proof}
This is a slightly more general analogue of~\cite[Proposition~4.6]{garner:small-object-argument}, we only detail the new aspects.
Let $u' \co \cal{I} \to \cal{M}_{\cart}$ denote the lift of $u$ through $I$.
By~\ref{adequate-u:colimits}, the left Kan extension $\Lan_u u' \co \cal{E}^\to \to \cal{M}_{\cart}$ exists and is computed pointwise as the colimit over $u'$ with weight sending $f$ to $\cal{E}^\to(u(-), f)$.
Furthermore, by preservation under $I$, the pointwise left Kan extension $U \defeq \Lan_u u$ exists and is computed as $\Lan_u u = I \cc \Lan_u u'$.
Comonad structure of $U$ and freeness follow as before.
Preservation of $\kappa$-transfinite compositions follows as in the end of the proof of~\cite[Proposition~4.22]{garner:small-object-argument}.
Preservation of $\overline{\cal{M}}$ is guaranteed by~\ref{adequate-u:preservation}.
\end{proof}

\begin{proposition} \label{reflecting-G2}
Let $(\cal{E}, \cal{M}, \kappa)$ be adequate.
Let $U$ be a comonad on $\cal{E}^\to$ such that $U$ lifts through the inclusion $\cal{M}_{\cart} \to \cal{E}^\to$.
Then $U$ has a reflection $V$ along the forgetful functor $\LAWFS(\cal{E}) \to \Cmd(\cal{E}^\to)$ and $V$ lifts through the inclusion $\cal{M}_{\cart} \to \cal{E}^\to$.

If $U$ preserves $\overline{\cal{M}}$, then so does $V$.
If $U$ preserves $\kappa$-transfinite compositions of maps in $\overline{\cal{M}}$, then so does $V$.
\end{proposition}

\begin{proof}
This is a slightly more general analogue of~\cite[Proposition~4.7]{garner:small-object-argument}, we only detail the new aspects.
By~\ref{adequate:pushout}, we still have an orthogonal factorization system $(L, R)$ with left class pushout squares and right class maps $(f, g)$ with $f$ an isomorphism, but only on the full subcategory of $\cal{E}^\to$ on $\cal{M}$.
Factoring the counit of $U$ as $U \to V \to \Id$ defines its domain-preserving reflection:
\[
\xymatrix{
  \bullet
  \ar[r]
  \ar[d]_{U}
&
  \dom
  \ar[d]^{V}
\\
  \bullet
  \ar[r]
&
  \bullet
  \pullback{ul}
}
\]
The comonad structure of $V$ and its freeness follow as before.
Note that $V$ lifts through $\cal{M}_{\cart} \to \cal{E}^\to$ by \cref{adequate-pushout}.

If $U$ preserves $\overline{\cal{M}}$, then so does $V$ by \cref{adequate-pushout}.
If $U$ preserves $\kappa$-transfinite compositions of maps in $\cal{N}$, then so does $V$ by commutativity of colimits.
\end{proof}

We briefly summarize the contents of \cite[Theorem~4.14]{garner:small-object-argument}.
The category $\FF(\cal{E})$ of functorial factorizations on $\cal{E}$ consists of sections of the composition functor $\cal{E}^{\to} \times_{\cal{E}} \cal{E}^{\to} \to \cal{E}^\to$.
There are two canonical monoidal structures $(\bot, \odot)$ and $(I, \otimes)$ on $\FF(\cal{E})$ forming a two-fold monoidal structure.
Garner's addition of the distributivity law to the theory of algebraic weak factorization systems ensures that the diamond
\[
\xymatrix{
&
  \AWFS(\cal{E})
  \ar[dl]
  \ar[dr]
\\
  \LAWFS(\cal{E})
  \ar[dr]
&&
  \RAWFS(\cal{E})
  \ar[dl]
\\&
  \FF(\cal{E})
}
\]
is precisely the diamond
\[
\xymatrix{
&
  \Bialg_{\otimes, \odot}(\FF(\cal{E}))
  \ar[dl]
  \ar[dr]
\\
  \Comon_{\odot}(\FF(\cal{E}))
  \ar[dr]
&&
  \Mon_{\otimes}(\FF(\cal{E}))
  \ar[dl]
\\&
  \FF(\cal{E})
}
\]
where
\begin{align*}
\Bialg_{\otimes, \odot}(\FF(\cal{E}))
&=
\Mon_{\otimes}(\Comon_{\odot}(\FF(\cal{E})))
\\&=
\Comon_{\odot}(\Mon_{\otimes}(\FF(\cal{E})))
\end{align*}
and all arrows are given by forgetful functors.

\begin{proposition} \label{reflecting-G1}
Let $(\cal{E}, \cal{M}, \kappa)$ be adequate.
Let $V$ be a domain-preserving comonad on $\cal{E}^\to$ such that $V$ lifts through the inclusion $\cal{M}_{\cart} \to \cal{E}^\to$ and preserves $\overline{\cal{M}}$ and $\kappa$-transfinite compositions of maps in $\overline{\cal{M}}$.
Then $V$ has a reflection $(L, R)$ along the forgetful functor $\AWFS(\cal{E}) \to \LAWFS(\cal{E})$.
Furthermore, $L$ lifts through the inclusion $\cal{M}_{\cart} \to \cal{E}^\to$ and $L, R$ preserve $\overline{\cal{M}}$.
\end{proposition}

\begin{proof}
This is a slightly more general analogue of~\cite[Proposition~4.21]{garner:small-object-argument}, we only detail the new aspects.
We need to construct a reflection of $V$ along the forgetful functor $\Mon_{\otimes}(\cal{V}) \to \cal{V}$ where $V \defeq \Comon_{\odot}(\FF(\cal{E}))$, thus construct the free monoid on the canonically pointed object $V$, noting that the unit $I$ of $\cal{V}$ is initial.
As per the proof~\cite[Proposition~4.18]{garner:small-object-argument}, the forgetful functor
\[
\xymatrix@C+1em{
  \cal{V}
  \ar[r]
&
  \FF(\cal{E})
  \ar[r]^-{d_0 \cc (-)}
&
  [\cal{E}^\to, \cal{E}^\to]
}
\]
preserves the relevant monoidal structure, creates connected colimits, and maps the pointed object $I \to V$ to $\eta_0 \co \Id \to R_0$ where $R_0(f)$ is the codomain part of the counit of $V$ on $f$ and $\eta_f = (V(f), \id_Y)$ for $f \co X \to Y$ in $\cal{E}$.
Instead of lifting $(\cal{V}, (L, R))$ to an object of $\ConfPM{\kappa}$, it thus suffices to lift $([\cal{E}^\to, \cal{E}^\to], R_0)$ to an object of $\ConfPM{\kappa}$, which will be accomplished by showing $(\cal{E}^\to, \overline{\cal{M}}, R_0)$ an object of $\ConfP{\kappa}$.
The result then follows by \cref{pointed-free-monoid}.

Let us verify the conditions of \cref{pointed-configurations}.
\begin{itemize}
\item
Condition~\ref{pointed-configurations:composition} is the lift of~\ref{adequate:composition} to the arrow category.
\item
Condition~\ref{pointed-configurations:pushout} is the lift of~\ref{adequate:composition} to the arrow category.
\item
Condition~\ref{pointed-configurations:unit} is satisfied by the assumption that $V$ is valued in $\cal{M}_{\cart}$.
\item
Condition~\ref{pointed-configurations:leibniz-application} is satisfied in view of \cref{pointed-configuration-simplification} by the lift of~\ref{adequate:binary-union} to the arrow category.
Note that $\eta$ is cartesian since $V$ is valued in $\cal{M}_{\cart}$.
\item
For condition~\ref{pointed-configurations:functor}, note that $R_0$ preserves $\overline{\cal{M}}$ if $\dom \cc R_0 = \cod \cc V$ and $\cod \cc R_0 = \cod$ map $\overline{\cal{M}}$ to $\cal{M}$.
But $V$ preserves $\overline{\cal{M}}$ by assumptions.
\item
For condition~\ref{pointed-configurations:convergence}, note that $R_0$ preserves a $\kappa$-transfinite compositions of maps in $\overline{\cal{M}}$ if $\dom \cc R_0 = \cod \cc V$ and $\cod \cc R_0 = \cod$ preserve them.
But $V$ preserves them by assumption.
\end{itemize}
We may thus apply \cref{pointed-free-monad}.
By part~\ref{pointed-free-monad:unit}, we have that $L$ preserves $\overline{\cal{M}}$.
By part~\ref{pointed-free-monad:functor}, we have that $R$ preserves $\overline{\cal{M}}$.
By part~\ref{pointed-free-monad:cartesian}, we have that the action of $L$ on morphisms is valued in pullback squares.
For this, we have to note that $(\cal{E}^\to, \overline{\cal{M}})$ has pullbacks and van Kampen pushouts lifted from $(\cal{E}, \cal{M})$.
\end{proof}

\begin{proposition} \label{reflecting}
Let $(\cal{E}, \cal{M}, \kappa, \cal{I}, u)$ be adequate.
Then the algebraic weak factorization system $(L, R)$ free on $\cal{I}$ exists and is algebraically free.
Furthermore, $L$ lifts through the inclusion $\cal{M}_{\cart} \to \cal{E}^\to$ and $L$ and $R$ preserve $\overline{\cal{M}}$.
\end{proposition}

\begin{proof}
For the existence of $(L, R)$ and the properties of $L$, we combine \cref{reflecting-G3,reflecting-G2,reflecting-G1}.
Algebraic freeness follows just like in~\cite{garner:small-object-argument}.
Note that $R$ preserves $\overline{\cal{M}}$ since $L$ does.
\end{proof}

\begin{remark}
In the setting of \cref{reflecting}, we furthermove have that the left category $\coalg(L) \to \cal{E}^\to$ lifts through $\cal{E}_{\cart}^\to \to \cal{E}^\to$, and even through $\cal{M}_{\cart} \to \cal{E}^\to$ if $\cal{M}$ is closed under retracts in the arrow category.
To see this, recall that copointed endofunctor coalgebras for $L$ are functorially retracts of free $L$-coalgebras and retract diagrams in the arrow category
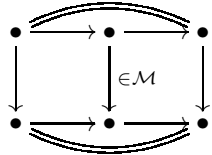
\[
\xymatrix{
  \bullet
  \ar[r]
  \ar[d]
  \ar@{=}@/^1em/[rr]
&
  \bullet
  \ar[r]
  \ar[d]^{\in \cal{M}}
&
  \bullet
  \ar[d]
\\
  \bullet
  \ar[r]
  \ar@{=}@/_1em/[rr]
&
  \bullet
  \ar[r]
&
  \bullet
}
\]
automatically have the left square a pullback as maps in $\cal{M}$ are mono.
\end{remark}

\section{Extension operations}

\subsection{Extension operations}

We now define what we mean by an extension operation in the algebraic context.
For this, let $v \co \cal{F} \to \cal{E}^\to$ be a category of arrows such that
\[
\xymatrix@C-1em{
  \cal{F}
  \ar[rr]^{v}
  \ar[dr]_{\cod \cc v}
&&
  \cal{E}^\to
  \ar[dl]^{\cod}
\\&
  \cal{E}
}
\]
is a map of Grothendieck fibrations from $\cod \cc v$ to $\cod$.
Given a category $\cal{B} \to \cal{E}^\to$ of arrows, we write $\cal{B}_{\cart} \to \cal{E}^\to_{\cart}$ for its restriction to cartesian squares, formally a pullback of categories.
In order to forestall confusion, note that $\cal{F}_{\cart}$ does not denote the wide subcategory of $\cal{F}$ restricted to cartesian arrows with respect to ${\cod} \cc v$, though it does coincide with it in case $v$ reflects cartesian arrows.

\begin{definition}[Extension operation] \label{extension}
Consider a category of arrows $u \co \cal{A} \to \cal{E}^\to$.
An \emph{extension operation} of $\cal{F}$ along $\cal{A}$ consists of the following data.
For any situation
\begin{equation}
\label{extension-property:0}
\begin{gathered}
\xymatrix{
  X
  \ar@{.>}[r]
  \ar[d]
  \pullback{dr}
&
  Y
  \ar@{.}[d]
\\
  A
  \ar[r]
&
  B
}
\end{gathered}
\end{equation}
with $A \to B$ in $\cal{A}$ and $X \to A$ in $\cal{F}$, we functorially have $Y \to B$ in $\cal{F}$ forming a pullback square as indicated that furthermore lifts in horizontal direction to a morphism in $\cal{F}^\to$.
This gives rise to a functor $\cal{F} \times_{\cal{E}} \cal{A} \to \cal{F}^\to$.
Furthermore, we require that this restricts to a functor $\cal{F}_{\cart} \times_{\cal{E}} \cal{A}_{\cart} \to \cal{F}^\to_{\cart}$.
\end{definition}

\begin{remark} \label{extension-functorial}
Let $\cal{A}_1 \to \cal{A}_2$ be a morphism in $\Cat/\cal{E}^\to$.
By functoriality, an extension operation of $\cal{F}$ along $\cal{A}_2$ gives rise to an extension operation of $\cal{F}$ along $\cal{A}_1$.
\end{remark}

\begin{remark} \label{extension-retract-closure}
Using base change, an extension operation of $\cal{F}$ along $\cal{A}$ gives rise to an extension operation of $\cal{F}$ along the retract closure $\overline{\cal{A}}$ of $\cal{A}$.
\end{remark}

\subsection{Lifting extension operations}

Let $(\cal{E}, \cal{M}, \kappa, \cal{I}, u)$ be adequate in the sense of \cref{adequate,adequate-u} and let $\cal{F} \to \cal{E}^\to$ be a category of arrows.
Recall the sequence
\begin{equation} \label{extension-sequence}
\begin{aligned}
\xymatrix{
  \AWFS(\cal{E})
  \ar[r]^{G_1}
&
  \LAWFS(\cal{E})
  \ar[r]^{G_2}
&
  \Cmd(\cal{E}^\to)
  \ar[r]^{G_3}
&
  \Cat/\cal{E}^\to
}
\end{aligned}
\end{equation}
of forgetful functors from the last section.
The free algebraic weak factorization system on $(\cal{I}, u)$ was constructed by stepwise reflection.
We will say for each of the four stages in~\eqref{extension-sequence} what it means to have \emph{extension} of $\cal{F}$ along it.
\begin{itemize}
\item
For $(\cal{I}, u) \in \Cat/\cal{E}^\to$, the notion of extension of $\cal{F}$ along $\cal{I}$ is precisely given by an extension operation in the sense of \cref{extension}.
\item
For $U \in \Cmd(\cal{E}^\to)$, the notion of extension of $\cal{F}$ along $U$ is given by an extension operation of $\cal{F}$ along $(\cal{E}^\to, U) \in \Cat/\cal{E}^\to$.
\item
For $V \in \LAWFS(\cal{E})$ a codomain-preserving comonad, the notion of extension of $\cal{F}$ along $\cal{I}$ is given by an extension operation of $\cal{F}$ along $(\cal{E}^\to, V) \in \Cat/\cal{E}^\to$.
\item
For $(L, R) \in \AWFS(\cal{E})$, the notion of extension of $\cal{F}$ along $(L, R)$ is given by an extension operation of $\cal{F}$ along $\coalg(L)$.
\end{itemize}

\begin{remark} \label{extension-eventual-image}
Let $U$ be a comonad on $\cal{E}^\to$.
Given $f \in \cal{E}^\to$, a copointed coalgebra structure on $f$ exhibits $f$ as a retract of $U(f)$.
This gives rise to a morphism in $\Cat/\cal{E}^\to$ from $\coalg(U)$ to the retract closure $\overline{(\cal{E}^\to, U)}$.
Conversely, there is a map from $(\cal{E}^\to, U)$ to $\Coalg(U)$ given by free coalgebras and a forgetful map from $\Coalg(U)$ to $\coalg(U)$.
In light of \cref{extension-functorial,extension-retract-closure}, extension operations of $\cal{F}$ along either of $\coalg(U)$, $\Coalg(U)$, $(\cal{E}^\to, U)$ are interderivable.
With this, we see for each of the four states in~\eqref{extension-sequence} that extension of $\cal{F}$ along an object is equivalent to the existence of an extension operation of $\cal{F}$ along its eventual image in $\Cat/\cal{E}^\to$.
\end{remark}

For each of the functors $G_1$ to $G_3$, we will respectively show in \cref{extending-G3,extending-G2,extending-G1} that the reflection construction of \cref{reflecting-G3,reflecting-G2,reflecting-G1} lifts to the level of extension.

\begin{proposition} \label{extending-G3}
Let $(\cal{E}, \cal{M}, \kappa, \cal{I}, u)$ be adequate.
Let $U \in \Cmd(\cal{E}^\to)$ denote the reflection of $(\cal{I}, u)$ along $\Cmd(\cal{E}^\to) \to \Cat/\cal{E}^\to$ given by \cref{reflecting-G3}.

Let $v \co \cal{F} \to \cal{E}^\to$ be a category of arrows such that $v$ is a map of Grothendieck fibrations from $\cod \cc v$ to $\cod$ and such that $\cal{F}_{\cart} \to \cal{E}^{\to}_{\cart}$ lifts colimits of shape $\cal{I}/f$ that are van Kampen for $f \in \cal{E}^\to$.
If $\cal{F}$ has extension along $(\cal{I}, u)$, then also along $U$.
\end{proposition}

\begin{proof}
Recall from \cref{reflecting-G3} that $U = \Lan_u u = I \cc \Lan_u u'$ where $u' \co \cal{I} \to \cal{M}_{\cart}$ was the lift of $u$ through $\cal{M}_{\cart} \to \cal{E}^\to$.
We abbreviate $u^s \defeq {\dom} \cc u$ and $u^t \defeq {\cod} \cc u$.

Let $f \co X \to Y$ be a map in $\cal{E}$ and let $v_a \co A \to \dom(U f)$ be an object of $\cal{F}$.
In order to extend the latter along
\begin{equation*}
U(f) \co \colim_{u(i) \to f} u^s(i) \to \colim_{u(i) \to f} u^t(i)
,\end{equation*}
we first pull it back along $u^s(i) \to \dom(U f)$ to an object $v_{a_{(i, h)}} \co A_{(i, h)} \to u^s(i)$ of $\cal{F}$ for each $h \co u(i) \to f$.
We then use the given extension operation to extend functorially in $(i, h)$ along $u(i)$ to produce an object $v_{b_{(i, h)}} \co B_{(i, h)} \to u^t(i)$ of $\cal{F}$.
By assumption, the cartesian squares $v_{a(i, h)} \to v_{a(i', h')}$ get mapped to cartesian squares $v_{b(i, h)} \to v_{b(i', h')}$.

Recalling that $u$ is valued in $\cal{M}_{\cart}$, we now have a diagram of shape $\braces{\bullet \to \bullet} \times \cal{I}/f$ in $\cal{M}_{\cart}$.
Using~\ref{adequate-u:colimits}, we may take its colimit, which is also a colimit in $\cal{E}^\to$, forming a pullback square
\[
\xymatrix{
  A
  \ar[r]
  \ar[d]^{v_a}
  \pullback{dr}
&
  B
  \ar[d]^{v_b}
\\
  \dom(U f)
  \ar[r]^{U f}
&
  \cod(U f)
\rlap{.}}
\]
Since the colimit for $\dom(U f)$ is van Kampen via~\ref{adequate-u:van-kampen}, the left map is the map $v_a$ we started with.
The right map $B \to \cod(U f)$ is the colimit of the maps $v_{b(i, h)} \co B_{(i, h)} \to u^t(i)$ in $\cal{E}^\to_{\cart}$ and thus lifts to a colimit $b \in \cal{F}_{\cart}$ as indicated since $\cal{F}_{\cart} \to \cal{E}^{\to}_{\cart}$ lifts colimits of shape $\cal{I}/f$.
Note that the coprojections $v_{b(i, j)} \to v_b$ are pullback squares since the colimit for $\cod(U f)$ is van Kampen via~\ref{adequate-u:van-kampen}.
Functoriality of colimits then induces the required lift of the above square to a morphism in $\cal{F}$.

The construction is evidently functorial.
To see that is maps cartesian inputs to cartesian outputs, use the van Kampen property.
\end{proof}

\begin{proposition} \label{extending-G2}
Let $(\cal{E}, \cal{M}, \kappa)$ be adequate.
Let $U \in \Cmd(\cal{E}^\to)$ satisfy the assumptions of \cref{reflecting-G2} and let $V$ denote its associated reflection along $\LAWFS(\cal{E}) \to \Cmd(\cal{E}^\to)$.

Let $v \co \cal{F} \to \cal{E}^\to$ be a category of arrows such that $v$ is a map of Grothendieck fibrations from $\cod \cc v$ to $\cod$ and such that $\cal{F}_{\cart} \to \cal{E}^\to_{\cart}$ lifts pushouts along maps in $\overline{\cal{M}}$.
If $\cal{F}$ has extension along $U$, then also along $V$.
\end{proposition}

\begin{proof}
This is a consequence of condition~\ref{adequate:pushout} that pushouts in $\cal{E}$ along $\cal{M}$ are van Kampen.
Recall the construction of $V f \co X \to C$ from $U f \co A \to B$ in the proof of \cref{reflecting-G2} where $f \co X \to Y$ is a map in $\cal{E}$:
\begin{equation} \label{extending-G2:0}
\begin{gathered}
\xymatrix{
  A
  \ar[r]
  \ar[d]_{U f}
&
  X
  \ar[d]^{V f}
\\
  B
  \ar[r]
&
  C
  \pullback{ul}
\rlap{.}}
\end{gathered}
\end{equation}
To extend $v_x \co X' \to X$ with $x \in \cal{F}$ along $V f$, we first pull it back along $A \to X$ to $v_a \co A' \to A$ using that $v \co \cal{F} \to \cal{E}^\to$ is a Grothendieck fibration.
We then use the given extension operation to extend $v_a$ along $U f$ to get $v_b \co B' \to B$.
The extension to $C$ then follows since
\[
\xymatrix{
  \cal{F}_{\cart}
  \ar[r]
&
  \cal{E}^\to_{\cart}
  \ar[r]^-{\cod}
&
  \cal{E}
}
\]
lifts pushouts along maps in $\cal{M}$ by assumption and \cref{van-kampen-pushout}.
The entire construction is functorial in maps from $f_1 \co X_1 \to Y_1$ to $f_2 \co X_2 \to Y_2$ (which get mapped to cartesian squares via $U$ and $V$) and $v_{x_1} \co X_1' \to X_1$ to $v_{x_2} \co X_2' \to X_2$ and furthermore results in a cartesian square if the square $v_{x_1} \to v_{x_2}$ is cartesian.
\end{proof}

The below proof crucially uses the fine-grained functoriality of free monad sequences in non-cocomplete settings devloped in \cref{free-monad-sequence}.

\begin{proposition} \label{extending-G1}
Let $(\cal{E}, \cal{M}, \kappa)$ be adequate.
Let $V \in \Cmd(\cal{E}^\to)$ satisfy the assumptions of \cref{reflecting-G1} and let $(L, R)$ denote its associated reflection along $\AWFS(\cal{E}) \to \LAWFS(\cal{E}^\to)$.

Let $\cal{F} \to \cal{E}^\to$ be a category of arrows such that $\cal{F}_{\cart} \to \cal{E}^{\to}_{\cart}$ lifts pushouts and $\alpha$-transfinite compositions of maps in $\overline{\cal{M}}$ for limit ordinals $\alpha \leq \kappa$.
If $\cal{F}$ has extension along $V$, then also along $(L, R)$.
\end{proposition}

\begin{proof}
In light of \cref{extension-eventual-image}, it will suffice to exhibit an extension operation of $\cal{F}$ along $(\cal{E}^\to, L)$, or equivalently along the domain part the unit $\eta$ of $R$.

Recall from \cref{reflecting-G1} via functoriality of \cref{pointed-free-monoid} that the monad $R$ is given by via \cref{pointed-free-monad} by the free monad sequence on the pointed functor $(R_0, \eta_0)$ on $\cal{E}^\to$ where $R_0 f$ is the codomain part of the counit of $V$ and $\eta_f = (V f, \id_Y)$ for $f \co X \to Y$ in $\cal{E}$.
Given to us is an extension operation of $\cal{F}$ along $(\cal{E}^\to, V)$, or equivalently along the domain part of the unit $\eta_0$ of $R_0$.

Let us define a new category $\cal{D}$ via the following pullback:
\[
\xymatrix{
  \cal{D}
  \ar[rr]^{U}
  \ar[d]
  \pullback{dr}
&&
  \cal{E}^\to
  \ar[d]^{\dom}
\\
  \cal{F}
  \ar[r]
&
  \cal{E}^\to
  \ar[r]_-{\cod}
&
  \cal{E}
\rlap{.}}
\]
Let $\cal{N}$ denote the class of morphisms in $\cal{D}$ that map to $\overline{\cal{M}}$ via $U$ and to pullback squares via $\cal{D} \to \cal{F} \to \cal{E}^\to$.
By assumptions and \cref{van-kampen-composition,van-kampen-pushout}, note that $(\cal{D}, \cal{N})$ has $\kappa$-compositions and pushouts and they are preserved by $U \co (\cal{D}, \cal{N}) \to (\cal{E}^\to, \overline{\cal{M}})$.
Lifting~\ref{adequate:binary-union}, we get that $(\cal{D}, \cal{N})$ has binary unions.

The setting $(\cal{D}, \cal{N})$ has been designed such that an extension operation of $\cal{F}$ along the domain part of the unit $\eta$ of a pointed endofunctor $F$ on $\cal{E}^\to$ where $\eta$ is valued in $\overline{\cal{M}}$ and $F$ preserves $\overline{\cal{M}}$ corresponds to a lift through $U$ of $(F, \eta)$ to a pointed endofunctor $(\hat{F}, \hat{\eta})$ on $\cal{D}$ where $\hat{\eta}$ is valued in $\cal{D}$ (since extension operations involve a cartesian square) and $\hat{F}$ preserves $\cal{N}$ (since the action on morphisms of an extension operation sends cartesian input to cartesian output).
We are given such a lift $(\hat{R}_0, \hat{\eta}_0)$ of $(R_0, \eta_0)$ and aim to produce a lift $(\hat{R}, \hat{\eta})$ of $(R, \eta)$.
Naturally, this will be an instance of functoriality of free monads.

Let us verify the remaining conditions of \cref{pointed-configurations} for $(\cal{D}, \cal{N}, \hat{R}_0)$.
Since $\eta_0$ is cartesian, so is $\hat{\eta}_0$.
Preservation by $\hat{R}_0$ of $\kappa$-transfinite compositions of maps in $\cal{N}$ is inherited from $R_0$.
Thus, \cref{pointed-free-monad} gives us the free monad $\hat{R}$ on $\cal{D}$ generated by $\hat{R}_0$.
By construction, $U$ induces a morphism from $(\cal{D}, \cal{N}, F_0)$ to $(\cal{E}^\to, \overline{\cal{M}}, R_0)$.
Functoriality of \cref{pointed-free-monad} hence tells us that $\hat{R}$ is a lift of $R$ to $\cal{D}$.
Thus we derive the required extension operation for $(R, \eta)$.
\end{proof}

\begin{definition}
Given adequate $(\cal{E}, \cal{M}, \kappa, \cal{I}, u)$, a category $v \co \cal{F} \to \cal{E}^\to$ of arrows is \emph{adequate} if
\[
\xymatrix@C-1em{
  \cal{F}
  \ar[rr]^{v}
  \ar[dr]_{\cod \cc v}
&&
  \cal{E}^\to
  \ar[dl]^{\cod}
\\&
  \cal{E}
}
\]
is a morphism of Grothendieck fibrations and $\cal{F}_{\cart} \to \cal{E}^\to_{\cart}$ lifts pushouts and $\alpha$-transfinite compositions for limit ordinals $\alpha \leq \kappa$ of maps in $\overline{\cal{M}}$ and colimits of shape $\cal{I}/f$ that are van Kampen for $f \in \cal{E}^\to$.
\end{definition}

Combinding \cref{extending-G3,extending-G2,extending-G1}, we derive our main theorem.

\begin{theorem} \label{extending}
Let $(\cal{E}, \cal{M}, \kappa, \cal{I}, u)$ be adequate.
Let $(L, R)$ denote the algebraic weak factorization system generated by it as per \cref{reflecting}.

Let $\cal{F} \to \cal{E}^\to$ be an adequate category of arrows.
An extension operation of $\cal{F}$ along $\cal{I}$ induces an extension operation of $\cal{F}$ along $\coalg(L)$.
\qed
\end{theorem}

Instantiating $(\cal{E}, \cal{M})$ to the class of monomorphisms in a Grothendieck topos, we derive the following corollary.

\begin{corollary} \label{extending-grothendieck-topos}
Let be given a Grothendieck topos $\cal{E}$, a functor $u \co \cal{I} \to \cal{E}^\to$ valued in monomorphisms and pullback squares with $\cal{I}$ small, and a functor $v \co \cal{F} \to \cal{E}^\to$ such that $v \co (\cal{F}, \cod \cc v) \to (\cal{E}^\to, \cod)$ is a morphism of Grothendieck fibrations over $\cal{E}$.

Assume that the colimit $(\Lan_u u)(f) = \colim_{u_i \to f} u_i$ is van Kampen for any $f \in \cal{E}^\to$ and that $\Lan_u u$ preserves morphisms in $\cal{E}^\to$ whose domain and codomain parts are monomorphisms.
Assume that $\cal{F}_{\cart} \to \cal{E}^\to_{\cart}$ lifts colimits that are van Kampen.

Let $(L, R)$ denote the algebraic weak factorization system cofibrantly generated by $\cal{I}$.
Then an extension operation of $\cal{F}$ along $\cal{I}$ induces an extension operation of $\cal{F}$ along $\coalg(L)$.
\end{corollary}

\begin{proof}
The topos $\cal{E}$ has pullbacks, pullbacks preserve monomorphisms~\ref{adequate:pullback}, and monomorphisms in $\cal{E}$ are closed under binary unions~\ref{adequate:binary-union}.
It is well-known by now that topoi are adhesive (monomorphisms in a topos are adhesive morphisms, \ie pushouts of monomorphisms are van Kampen~\ref{adequate:pushout}).
A first attempt of a constructive proof was given in \cite{lack:toposes-adhesive}, with a complete, though non-constructive proof in \cite{lack:embedding-adhesive}.
To our knowledge, the first full constructive proof can be found in \cite{garner-lack:adhesive}.
Similarly, Grothendieck topoi are exhaustive (transfinite compositions of monomorphisms are van Kampen~\ref{adequate:composition}).
\notec{Insert reference}

We use local presentability of $\cal{E}$ to satisfy to satisfy~\ref{adequate-u:compact}.
The remaining conditions of \cref{extending} are satisfied by assumption. 
\end{proof}

\notec{Add examples and discuss the connection with classifying objects.}

\notec{Add more relevant citations.}

\bibliographystyle{plain}
\bibliography{../../common/uniform-kan-bibliography}

\end{document}